\documentclass[11pt,oneside,english]{amsart}
\usepackage{bbm}
\usepackage{mathrsfs}
\usepackage{amsfonts}
\usepackage{amssymb}
\usepackage[colorlinks]{hyperref}
\usepackage{graphicx}
\usepackage{babel}

\makeatletter \theoremstyle{plain}
 \newtheorem{thm}{Theorem}[section]
 
 \newtheorem{prop}[thm]{Proposition}
 
 \numberwithin{equation}{section}
 \numberwithin{figure}{section}
 \theoremstyle{plain}
 \theoremstyle{definition}

\newcommand{\cH}{{{\mathcal H}}}
\newcommand{\fH}{{{\mathfrak H}}}

\newcommand{\bH}{{{\bf H}}}

\newcommand{\C}{{{\mathbb C}}}

\newcommand{\rd}{{{\rm d}}}
\makeatother

\begin{document}
\title [Bisectors in the Heisenberg group II]
{Bisectors in the Heisenberg group II}

\author{Gaoshun Gou}
\address {School of Science,
Chongqing-University of Posts and Telecommunications,
Chongqing {\rm 400065},
P. R. China.}
\email{gougs@cqupt.edu.cn}

\author{Joonhyung Kim}
\address{Department of Mathematics Education,
Chungnam National University,
Daejeon 34134, South Korea.}
\email{calvary@cnu.ac.kr}

\author{Ioannis D. Platis}
\address{Department of Mathematics,
University of Patras,
University Campus, 26504 Rion, Achaia,
Greece.}
\email{idplatis@upatras.gr}

\author{Li-Jie Sun}
\address{Department of Mathematics,
Osaka Metropolitan University,
3-3-138 Sugimoto, Sumiyoshi-ku, Osaka 558-8585, Japan.}
\email{sun@omu.ac.jp}


\begin{abstract}
Following previous work on metric bisectors with respect to the Kor\'anyi metric, we study here bisectors in the Heisenberg group $\mathfrak{H}$ equipped with the Carnot--Carath\'eodory metric $d_{cc}$.
\end{abstract}

\maketitle

\section{Introduction}

A metric bisector in a metric space $(X, d)$ is defined as the subset of points equidistant from two distinct anchor points $p_1, p_2 \in X$. While bisectors in classical Riemannian manifolds such as Euclidean space are structurally simple (for example, hyperplanes), they can manifest as highly complex geometric objects in arbitrary metric spaces. This complexity is prominently featured within the Heisenberg group $\mathfrak{H}$, a fundamental example of a two-step nilpotent Lie group and a simple sub-Riemannian manifold \cite{CCG2007, M02}.

The Heisenberg group can be naturally equipped with two bi-Lipschitz equivalent, left-invariant metrics: the algebraically defined Kor\'anyi--Cygan metric $d_K$, and the Carnot--Carath\'eodory path metric $d_{cc}$. A critical geometric bridge between these two spaces is that both $d_K$ and $d_{cc}$ share identical isometry groups and many similarities. Consequently, the normalisation techniques applied to anchor points, that is, categorising them into cases where they lie on either the same infinite or finite $\mathbb{C}$-circle, remain universally applicable across both metrics. However, despite these shared symmetries, we will show that bisectors with respect to $d_{cc}$ are generally not the same objects as bisectors with respect to $d_K$.

In our previous work \cite{GJP}, we established a complete geometric classification of bisectors with respect to the Kor\'anyi metric. We demonstrated that a Kor\'anyi bisector is always a spinal sphere, and conversely, every spinal sphere acts as a Kor\'anyi bisector \cite{Gol, PP06}. When examining their horizontal geometry, a stark dichotomy emerged based on the configuration of the focal points. For anchor points on an infinite $\mathbb{C}$-circle (the vertical case), the Kor\'anyi bisector $\mathcal{B}_K^0$ is a complex plane $\mathbb{C}$ and is horizontally minimal everywhere. Conversely, for points on a finite $\mathbb{C}$-circle (the horizontal case), the bisector $\mathcal{B}_K^1$ forms the spinal sphere $\mathfrak{S}_1$, governed by the cubic polynomial surface $x(x^2+y^2+1)-yt = 0$. This algebraic surface fails to be globally horizontally minimal; its horizontal mean curvature diverges to infinity near the characteristic locus $(0, \pm1, 0)$ and tends to zero solely as it moves away from these singular points.

In the present paper, we extend this investigation to bisectors defined by the Carnot--Carath\'eodory metric, uncovering both remarkable parallels and structural divergences. Because $d_{cc}$ is an intrinsic path metric defined by the infimum of horizontal lengths of all horizontal curves joining two points \cite{CDPT}, transitioning from $d_K$ to $d_{cc}$ inherently replaces algebraic gauge equations with transcendental systems governed by Hamiltonian geodesic flows \cite{CCG2007}.

Our comparative analysis yields the following central results:

\begin{enumerate}
    \item \textbf{Symmetrical Equivalency in the Vertical Case:} For focal points positioned on an infinite $\mathbb{C}$-circle, the structural equivalence between the two metrics is perfectly preserved. The Carnot--Carath\'eodory bisector $\mathcal{B}_{cc}^0$ is identically the complex plane $t=0$. Like its Kor\'anyi counterpart, it is devoid of characteristic singularities and maintains global horizontal minimality.
    \item \textbf{Transcendental Divergence in the Horizontal Case:} The primary geometric divergence materialises for focal points on a finite $\mathbb{C}$-circle. Due to the strict lack of the inversion transformation in the similarity group of $d_{cc}$, we cannot assume a generic horizontal normalisation as in the Kor\'anyi case. The resulting Carnot--Carath\'eodory bisector $\mathcal{B}_{cc}^1$ is defined by a transcendental equality requiring the inversion of the Reeb angle associated with horizontal geodesics. Although both $\mathcal{B}_K^1$ and $\mathcal{B}_{cc}^1$ share the same sub-Riemannian phenomenon---specifically, the divergence of the horizontal mean curvature $H^h$ at the characteristic locus---their global topologies differ.
    \item \textbf{Complex Hyperbolic Implications:} Finally, this dichotomy highlights the varying relationships these metrics share with the complex hyperbolic plane $\mathbf{H}_{\mathbb{C}}^2$. Kor\'anyi bisectors seamlessly bound significant submanifolds in $\mathbf{H}_{\mathbb{C}}^2$ because the Kor\'anyi gauge naturally arises on the boundary $\partial\mathbf{H}_{\mathbb{C}}^2$ via the Busemann functions of the Bergman metric \cite{Gol, PP10}. In contrast, the Carnot--Carath\'eodory bisectors for finite $\mathbb{C}$-circles abandon the projective algebraic structure in favour of the localised, non-holonomic constraints of the Heisenberg distribution, remaining rigidly isolated to the intrinsic sub-Riemannian mechanics of the boundary itself.
\end{enumerate}

The remainder of this paper is organised as follows. Section 2 provides the preliminary geometric and algebraic framework of the Heisenberg group, including the Carnot--Carath\'eodory metric and the intrinsic characterisation of $\mathbb{C}$-circles. In Section 3, we address the normalisation of anchor points, highlighting the obstruction caused by the absence of an inversion isometry, and establish the bisector equations for both the vertical and horizontal cases. Section 4 computes the horizontal mean curvature to determine sub-Riemannian minimality, presenting explicit analytic and numerical approximations for the curvature of generic bisectors. Finally, Section 5 contextualises our findings within the broader framework of complex hyperbolic geometry.

\section{Preliminaries}

This section establishes the geometric and algebraic framework of the Heisenberg group $\fH$, which is the foundational space of our study. We first review the sub-Riemannian and Carnot--Carath\'eodory metrics, alongside their respective similarity groups, before providing a strictly intrinsic characterisation of $\C$-circles constructed directly from the contact structure of $\fH$.
\subsection{The Heisenberg Group}\label{sec:heisen}
The boundary of the complex hyperbolic plane, denoted $\partial\bH^2_\mathbb{C}\setminus\{\infty\}$, is in natural bijection with the Heisenberg group $\fH$. This space is defined as the set $\mathbb{C}\times\mathbb{R}$ equipped with the group multiplication $*$ given by
\begin{equation}
(z,t)*(w,s)=(z+w,t+s+2\Im(z\overline{w})),
\end{equation}
for every $(z,t)$ and $(w,s)$ in $\fH$. There are two natural, left-invariant metrics defined on $\fH$. The first is the \textit{Kor\'anyi--Cygan metric}, given by
\begin{equation}
\mathrm{d}_{K}((z,t),(w,s))=|(z,t)^{-1}*(w,s)|_{K},
\end{equation}
where $|\cdot|_K$ is the Kor\'anyi gauge defined as
\begin{equation}
|(z,t)|_{K}=\left||z|^2-it\right|^{1/2}.
\end{equation}
The similarity group ${\rm Sim}_K(\fH)$ of $\fH$ with respect to the Kor\'anyi--Cygan metric comprises the following transformations:
\begin{enumerate}
\item[(i)] Left translations $L_p$, $p\in\fH$, defined by $L_p(q)=p*q$.
\item[(ii)] Rotations around the vertical axis $R_\theta$, $\theta\in\mathbb{R}$, defined by $R_\theta(z,t)=(e^{i\theta}z,t)$.
\item[(iii)] Dilations $D_\delta$, $\delta>0$, defined by $D_\delta(z,t)=(\delta z,\delta^2 t)$.
\end{enumerate}
Left translations, rotations, and conjugation form the isometry group of $\fH$ for $\rd_K$. The similarity group ${\rm Sim}_\mathfrak{H}$ may be viewed as the isotropy subgroup of $\infty$ in ${\rm SU}(2,1)$ (see for example \cite{PP10}), giving
$$
{\rm Sim}_K(\fH)=\fH\rtimes({\rm U}(1)\times {\mathbb R}_{+}).
$$

The Kor\'anyi inversion
	\begin{equation}\label{formula:inv}
	R(z,t)
	=
	\left(
	\frac{z}{-|z|^2+it},
	\frac{-t}{|-|z|^2+it|^2}
	\right)
	\end{equation}
	satisfies
	\[
	|R(p)|_K
	=
	\frac1{|p|_K}
	\]
	and
	\[
	d_K(R(p),R(q))
	=
	\frac{d_K(p,q)}
	{|p|_K\cdot |q|_K}.
	\]
See \cite{KR} for reference. The second identity is usually regarded as the fundamental metric property of the Korányi inversion, completely analogous to classical Euclidean inversion.  Also, it's well known that the similarity group together with the Kor\'anyi inversion gives the Bruhat decomposition of $\rm PU(2, 1),$ which acts diagonally transitively on $\partial\bH^2_\mathbb{C}.$

The second metric relates to the sub-Riemannian structure of $\fH$. The Heisenberg group is a 2-step nilpotent Lie group, and its Lie algebra $\mathfrak{h}$ admits the decomposition $\mathfrak{h}=V^1\oplus V^2$. The horizontal distribution $V^1$ is spanned by the left-invariant vector fields:
$$
X=\frac{\partial}{\partial x}+2y\frac{\partial}{\partial t},\quad
Y=\frac{\partial}{\partial y}-2x\frac{\partial}{\partial t},
$$
and $V^2$ is spanned by the vertical vector field $T=\frac{\partial}{\partial t}$. The standard contact structure on the Heisenberg group $\fH$ is given by the kernel of the 1-form
$$
\omega = d t + 2(x\,d y - y\,d x) = d t + 2\Im(\overline{z}\,d z).
$$
It follows directly that
$$
\omega(X)=\omega(Y)=0, \qquad \omega(T)=1.
$$
Also $d\omega(T,\cdot)=0$ and thus $T$ is the Reeb vector field associated to $\omega$. Since $\{X,Y,T\}$ forms a global frame for the tangent bundle $T\fH$, the contact form $\omega$ is uniquely determined by these conditions. For each point $p\in \fH$, $V^1_p=\cH_p(\fH)$ represents the \textit{horizontal tangent space of $\fH$ at $p$}. We define a sub-Riemannian metric $(\cdot,\cdot)_h$ on the horizontal distribution by setting
$$
(X,X)_h=(Y,Y)_h=1,\qquad (X,Y)_h=(Y,X)_h=0.
$$
A smooth curve $\gamma:[a,b]\to\fH$, given by $\gamma(s)=(z(s),t(s))$, is called a \textit{horizontal curve} if $\dot{\gamma}(s)\in \cH_{\gamma(s)}(\fH)$ for all $s\in[a,b]$. Equivalently, this imposes the non-holonomic constraint:
$$
\dot{t}(s)=-2\Im\left(\overline{z(s)}\dot{z}(s)\right).
$$
The \textit{horizontal length} of a smooth rectifiable horizontal curve $\gamma$ is given by
$$
\ell_h(\gamma)=\int_a^b|\dot{\gamma}_h(s)|_h{\rm d}s=\int_a^b\left[\left(\dot{\gamma}(s),X_{\gamma(s)}\right)_h^2+
\left(\dot{\gamma}(s),Y_{\gamma(s)}\right)_h^2\right]^{1/2}{\rm d}s=\int_a^b|\dot z(s)|\,\rd s.
$$
The \textit{Carnot--Carath\'eodory (CC)} distance $\mathrm{d}_{cc}(p,q)$ between any two points $p,q\in\fH$ is defined as the infimum of the horizontal lengths of all horizontal curves joining $p$ and $q$.
By the left-invariance of the metric, $d_{cc}(p, q) = \| p^{-1} * q \|_{cc}$, where $\| \cdot \|_{cc}$ is the CC gauge (norm): utilising the Hamiltonian formulation of the sub-Riemannian geodesic problem (see \cite{CCG2007}), the horizontal geodesics connecting the origin to a point $(z,t)$ can be parametrised by the distance $s$ and a Reeb angle parameter $\phi \in [-\pi, \pi]$:
\begin{equation} \label{eq:param}
|z|^2 = \frac{s^2}{\phi^2} \sin^2(\phi), \quad t = \frac{s^2}{2\phi^2} (\phi - \sin(\phi)\cos(\phi)).
\end{equation}
The Reeb parameter $\phi$ is uniquely determined as the solution to the transcendental relation:
\begin{equation} \label{eq:reeb_ratio}
\frac{|z|^2}{t} = \frac{2\sin^2(\phi)}{\phi - \sin(\phi)\cos(\phi)}.
\end{equation}
Consequently, the CC norm is given explicitly in terms of $\phi$ by \cite{CCG2007}:
\begin{equation} \label{eq:cc_norm_explicit}
\|(z,t)\|_{cc} = \frac{|z| \, \phi}{\sin(\phi)} = \sqrt{\frac{2 |t| \, \phi^2}{|\phi - \sin(\phi)\cos(\phi)|}}.
\end{equation}
For the relationship between the Kor\'anyi gauge and the CC norm see for instance \cite[Theorem 2]{AS} or \cite{CDPT}.
We also note that the following hold.
\begin{itemize}
\item The metrics $\mathrm{d}_{K}$ and $\mathrm{d}_{cc}$ are bi-Lipschitz equivalent; however, $\mathrm{d}_{cc}$ is a path metric, whereas $\mathrm{d}_{K}$ is not.
\item The similarity group ${\rm Sim}_{K}(\fH)$ and the similarity group ${\rm Sim}_{CC}(\fH)$ are same.
\end{itemize}

\subsection{$\mathbb{C}$-circles}
\label{sec:c_circles}
A $\mathbb{C}$-circle in the Heisenberg group $\mathfrak{H}$ is a distinguished one-dimensional curve that is intertwined with the sub-Riemannian contact structure of the space. In the broader context of CR manifolds, these curves are frequently referred to as \textit{chains}. Intrinsically, rather than relying on their classical definition as the asymptotic boundaries of complex geodesics \cite{Gol}, $\mathbb{C}$-circles can be classified into two specific geometric categories based solely on their trajectories within $\mathfrak{H}$.

\smallskip

\noindent{\it Infinite $\mathbb{C}$-circles:}
An infinite $\mathbb{C}$-circle is an integral curve of the Reeb vector field $T = \partial/\partial t$. Geometrically, these correspond to strictly vertical straight lines in $\mathfrak{H}$. Given a fixed horizontal projection point $z_0 \in \mathbb{C}$, an infinite $\mathbb{C}$-circle is globally parametrised by the free parameter $t \in \mathbb{R}$:
\begin{equation*}
z(t) = z_0, \quad t \in \mathbb{R}.
\end{equation*}
Because they flow purely along the Reeb vector field, infinite $\mathbb{C}$-circles are everywhere orthogonal to the horizontal contact distribution $V^1$.

\smallskip

\noindent {\it Finite $\mathbb{C}$-circles:}
A finite $\mathbb{C}$-circle is a compact, unknotted closed curve forming a planar ellipse when viewed in $\mathbb{R}^3 \cong \mathfrak{H}$. Its defining characteristic is that its projection onto the horizontal complex plane $\mathbb{C}$ is a standard Euclidean circle, while its vertical displacement is rigidly locked to a transverse lifting condition.

Section 4.3 of \cite{Gol} discusses finite chains (finite $\mathbb{C}$-circles in our terminology). For the convenience of the reader, we recall here a more explicit description of finite $\mathbb{C}$-circles. 
\begin{prop}	
	Let
	\[
	C_0=\{(z,0)\in\mathfrak H:|z|=1\}.
	\]
Every finite $\mathbb{C}$-circle is a ${\rm Sim}(\mathfrak H)$-image of $C_0.$
\end{prop}
\begin{proof}
Let $C$ be a finite $\mathbb{C}$-circle and the projected Euclidean circle have a centre $z_0 = x_0 + iy_0 \in \mathbb{C}$ and a radius $r > 0$. Let $t_0 \in \mathbb{R}$ dictate the base vertical height.  We claim that any finite $\mathbb{C}$-circle is 
 uniquely parametrised by the angle $\theta \in [0, 2\pi)$ as follows:
\begin{equation*}
\begin{cases}
x(\theta) = x_0 + r\cos\theta \\
y(\theta) = y_0 + r\sin\theta \\
t(\theta) = t_0 - 2r(x_0\sin\theta - y_0\cos\theta).
\end{cases}.
\end{equation*}
%
Indeed, by definition, $C$ is the boundary of a
	complex geodesic $L$. Using the standard Hermitian form
	\[
	\langle Z,W\rangle
	=
	Z_1\overline{W_3}
	+
	Z_3\overline{W_1}
	+
	Z_2\overline{W_2},
	\]
	the complex geodesic $L$ is the intersection of
	$\mathbf H^2_{\mathbb C}$ with a complex projective line. Hence there
	exists a positive vector $
	v=[a\quad b\quad c]^t$
	such that
	\[
	L=\{[Z]\in\mathbf H^2_{\mathbb C}:\langle Z,v\rangle=0\}.
	\]
	A boundary point $(z,t)\in\mathfrak H$ is represented by
	\[
	\hat p(z,t)
	=
	\begin{bmatrix}
		-\dfrac{|z|^2-it}{2}\\[2mm]
		z\\[2mm]
		1
	\end{bmatrix}.
	\]
	Therefore $(z,t)\in C=\partial_\infty L$ if and only if
	\[
	\langle \hat p(z,t),v\rangle=0.
	\]
	Substituting the above lift yields
	\[
	-\frac{|z|^2-it}{2}\,\overline c
	+z\overline b
	+\overline a
	=0.
	\]
Since $C$ is finite, the corresponding complex geodesic does not pass
	through $\infty$. Hence $c\neq0$ and we may divide by $\overline c$:
	\[
	-|z|^2+it
	+2z\frac{\overline b}{\overline c}
	+2\frac{\overline a}{\overline c}
	=0.
	\]
Setting $z_0=b/c, A=2\overline a/\overline c, $ we have
	\[
	-|z|^2+it+2z\overline{z_0}+A=0.
	\]
Writing $A=\alpha+i\beta,$ and separating real and imaginary parts, we obtain
	\[
	|z|^2-2\Re(z\overline{z_0})=\alpha,
	\qquad
	t+\beta
	=
	-2\Im(z\overline{z_0}).
	\]
Thus, we obtain
	\[
	|z-z_0|=r, \qquad  t=t_0-2\Im(z\overline{z_0}),
	\]
by setting $r=\sqrt{\alpha+|z_0|^2},\, t_0=-\beta.$
Consequently,
	\[
	C
	=
	\Bigl\{
	(z,t)\in\mathfrak H:
	|z-z_0|=r,\;
	t=t_0-2\Im(z\overline{z_0})
	\Bigr\},
	\]
	as desired.
	
Finally, our result follows immediately by 
	\[
	C=
	L_{(z_0,t_0)}\circ D_r(C_0).
	\]
\end{proof}
Unlike arbitrary lifts, a finite $\mathbb{C}$-circle remains strictly transverse to the contact distribution of the Heisenberg group at all points, reflecting its rigid sub-Riemannian character.

\section{Normalisation and Carnot--Carath\'eodory Bisectors}
We now address the normalisation of anchor points required to define Carnot--Carath\'eodory bisectors. We detail the structural differences between the vertical and horizontal cases.

\subsection{Points on an Infinite $\mathbb{C}$-circle: the Vertical Bisector}
If the focal points lie on an infinite $\mathbb{C}$-circle, they can be mapped via similarities (translations and dilations, which are valid in ${\rm Sim}_{CC}(\mathfrak{H})$) to the standard vertical pair $p_1 = (0,-1)$ and $p_2 = (0,1)$.
Equating the CC distances to a generic point $p=(z,t)$ yields $$\| (z, t+1) \|_{cc} = \| (z, t-1) \|_{cc}.$$
Based on \eqref{eq:cc_norm_explicit}, for a fixed horizontal radius $r = |z|$, the distance function is strictly monotonically increasing with respect to $|t|$. Thus, $|t+1| = |t-1|$, which requires $t=0$.
We obtain the following:
\begin{prop}
The vertical Carnot--Carath\'eodory bisector $\mathcal{B}_{cc}^0$ is exactly the complex plane $C_0 = \{ (z,t) \in \mathfrak{H} \mid t = 0 \}$.
\end{prop}

%
%

\subsection{Points on a Finite $\mathbb{C}$-circle: The General Bisector}
To analyze the bisector explicitly, we consider another normalization of two points $p_1$ and $p_2$ lying on a finite $\mathbb{C}$-circle.
We can apply a left translation to map $p_1$ to the origin $(0,0)$, mapping $p_2$ to some point $(z,t)$ with $z=x+iy \neq 0$. A rotation can align $z$ with the positive real axis. However, any subsequent dilation $D_\lambda$ scales the coordinates to $(\lambda z, \lambda^2 t)$. Notice that the continuous geometric invariant $c = \frac{t}{|z|^2}$ remains strictly preserved under this dilation, since $\frac{\lambda^2 t}{|\lambda z|^2} = \frac{t}{|z|^2}$. Thus, attempting to force the anchor points into the purely horizontal generic position (where $t=0$ and thus $c=0$) breaks the intrinsic sub-Riemannian constraints of the space, unless $c=0$ initially.  We arrive at the standard symmetrically normalised anchor points:
\begin{equation}\label{eq:normal_horizontal}
q_1 = (-1, -c) \quad \text{and} \quad q_2 = (1, c),
\end{equation}
where $c = \frac{t}{|z|^2}$. The purely horizontal configuration corresponds strictly to the singular case where $c = 0$, which we cannot universally assume.

%
%

The general Carnot--Carath\'eodory bisector is defined as $\mathcal{B}_{cc}^k = \{ p \in \mathfrak{H} \mid d_{cc}(q_1, p) = d_{cc}(q_2, p) \}$. The transformation $\Phi(z, t) = (-\overline{z}, -t)$ is an isometry of $\mathfrak{H}$ that swaps the focal points. Consequently, $\mathcal{B}_{cc}^k$ is invariant under $\Phi$, and the $y$-axis is always completely contained within the bisector.

Let $S(r, h)$ denote the Carnot--Carath\'eodory distance from the origin to a point with horizontal radius $r$ and vertical height $h$. The bisector equation for $\mathcal{B}_{cc}^k$ is governed by the implicit transcendental surface:
\begin{equation} \label{eq:bisector_gen}
S_1=S\left( \sqrt{(x+1)^2 + y^2}, t + c - 2y \right) = S\left( \sqrt{(x-1)^2 + y^2}, t - c + 2y \right)=S_2.
\end{equation}

\subsection{Analytic Divergence: Kor\'anyi vs. Carnot--Carath\'eodory}
The fundamental analytic distinction between the two metrics emerges here: the Kor\'anyi norm is explicitly algebraic; therefore, when evaluating the equidistant locus for the horizontal case, the algebraic formulation naturally yields the standard Kor\'anyi spinal sphere $\mathfrak{S}_1$, governed by the cubic equation $x(x^2+y^2+1)-yt = 0$~\cite{GJP}.

Conversely, the distance function $S(|z|, t)$ for the CC metric relies on the Reeb angle $\phi$. Because $\phi$ is determined by the transcendental relationship in \eqref{eq:reeb_ratio}, the bisector equation $S_1 = S_2$ in \eqref{eq:bisector_gen} is highly non-linear and strictly transcendental. It cannot be reduced to any finite algebraic polynomial. The sub-Riemannian growth rate of $d_{cc}$ dictates that the transcendental surface $\mathcal{B}_{cc}^1$ pinches more tightly along the vertical axis than its algebraic counterpart $\mathcal{B}_K^1$.

\section{Horizontal Mean Curvature}

We now delve into the differential geometry of CC bisector surfaces. We compute their horizontal mean curvature $H^h$ to determine sub-Riemannian minimality, explicitly contrasting the globally minimal vertical bisectors with the singular behaviour of generic CC bisectors, and applying Taylor expansions to approximate the non-linear curvature.

To evaluate the horizontal mean curvature $H^h$ of these bisector surfaces, we represent the bisector as an implicit surface defined by a smooth function $F(x,y,t) = 0$. The horizontal normal vector field is given by $$N_h = \frac{\nabla_h F}{|\nabla_h F|_h},\quad \nabla_hF=XFX+YFY.$$ Applying the sub-Riemannian divergence theorem yields:
\begin{equation} \label{eq:meancurv}
2H^h = \text{div}_h(N_h) = \frac{(YF)^2 XXF + (XF)^2 YYF - XF \cdot YF(XYF + YXF)}{[(XF)^2 + (YF)^2]^{3/2}}.
\end{equation}

\subsection{Minimality of the Vertical Bisector}
For clarity we show that the vertical bisector $\mathcal{B}_{cc}^0$ defined  by $F(x,y,t) = t = 0$ is minimal. We first have:
$$
    XF = X(t) = 2y,\quad
    YF = Y(t) = -2x.
$$
The denominator in \eqref{eq:meancurv} evaluates to $[(2y)^2 + (-2x)^2]^{3/2} = 8(x^2 + y^2)^{3/2}$, which is strictly non-zero away from the origin; the latter is the {\it characteristic point} of $\mathcal{B}_{cc}^0$. Recall that the {\it characteristic locus} of a hypersurface $F=0$ embedded in $\fH$ comprises points (the {\it characteristic points}) at which the horizontal velocity $|\nabla_hF|$ vanishes. At the characteristic locus, the denominator in \eqref{eq:meancurv} vanishes, causing $H^h$ to diverge to infinity, an intrinsic hallmark of sub-Riemannian surfaces \cite{CDPT}.
Taking the second-order horizontal derivatives yields:
$$
XXF = X(2y) = 0,\,\,
    YYF = Y(-2x) = 0,\,\,
    XYF = X(-2x) = -2,\,\,
    YXF = Y(2y) = 2.
$$
Substituting these values directly into the numerator of $\mathcal{B}_{cc}^0$ \eqref{eq:meancurv} results in:
\begin{equation}
(-2x)^2(0) + (2y)^2(0) - (2y)(-2x)(-2 + 2) = 0.
\end{equation}
 Thus $H^h \equiv 0$, proving that the vertical CC bisector is globally horizontally minimal.

\subsection{Horizontal Mean Curvature of the General Bisectors}
For the general bisector $\mathcal{B}_{cc}^k$, the implicit function takes the form $F(x,y,t) = S_1 - S_2 = 0$. Unlike the polynomial structure of the Kor\'anyi spinal sphere, this equation relies on the numerical inversion of the Reeb parameter $\phi$. This structural difference generates a non-vanishing horizontal mean curvature $H^h$ away from the characteristic locus due to the sub-Riemannian focusing of horizontal geodesics.

\subsection{Taylor Expansion and Numerical Approximation for $H^h$}
Given that the implicit function $F(x,y,t) = S_1 - S_2 = 0$ defining $\mathcal{B}_{cc}^k$ relies fundamentally on the non-algebraic inversion of the Reeb parameter $\phi$, expressing the horizontal mean curvature $H^h$ globally in closed form is intractable. To analyse the curvature behaviour analytically and globally, we employ local Taylor expansions and finite difference schemes adapted to the sub-Riemannian geometry.

\subsubsection{Taylor Expansion for Small Vertical Displacements}
For regions of the bisector where the vertical height $h$ is small relative to the horizontal radius $r$ (corresponding to the regime where the Reeb angle $\phi \to 0$), we can approximate the Hamiltonian geodesic parametrisation via a Taylor series. Expanding \eqref{eq:param} yields:
\begin{equation}
r \approx S \left(1 - \frac{\phi^2}{6}\right), \qquad h \approx \frac{S^2}{3}\phi.
\end{equation}
Isolating $\phi \approx \frac{3h}{r^2}$ produces an explicit algebraic Taylor expansion for the Carnot--Carath\'eodory distance:
\begin{equation}\label{eq:asymp_dist}
S(r,h) \approx r + \frac{3h^2}{2r^3}.
\end{equation}
Substituting \eqref{eq:asymp_dist} into the bisector equation $F(x,y,t) = S_1 - S_2 = 0$ yields an explicitly rational surface approximation $F_{\text{Taylor}}(x,y,t) = 0$:
\begin{equation} \label{eq:F_asympt}
F_{\text{Taylor}}(x,y,t) = r_1 - r_2 + \frac{3}{2}\left( \frac{(t + c - 2y)^2}{r_1^3} - \frac{(t - c + 2y)^2}{r_2^3} \right) = 0,
\end{equation}
where $r_1 = \sqrt{(x+1)^2 + y^2}$ and $r_2 = \sqrt{(x-1)^2 + y^2}$.

Specifically, letting $H_1 = t + c - 2y$ and $H_2 = t - c + 2y$, we have $X(H_1) = X(H_2) = 2y$, $Y(H_1) = -2(x+1)$, and $Y(H_2) = -2(x-1)$. By directly applying the horizontal vector fields $X$ and $Y$ to $F_{\text{Taylor}}$, we obtain explicit algebraic approximations for the horizontal gradient ($\nabla_h F_{\text{Taylor}}$) given by:
\begin{align}
XF_{\text{Taylor}} &= \frac{x+1}{r_1} - \frac{x-1}{r_2} + 6y \left(\frac{H_1}{r_1^3} - \frac{H_2}{r_2^3}\right) - \frac{9}{2} \left( \frac{H_1^2(x+1)}{r_1^5} - \frac{H_2^2(x-1)}{r_2^5} \right), \\
YF_{\text{Taylor}} &= \frac{y}{r_1} - \frac{y}{r_2} - 6 \left(\frac{H_1(x+1)}{r_1^3} - \frac{H_2(x-1)}{r_2^3}\right) - \frac{9}{2}y \left( \frac{H_1^2}{r_1^5} - \frac{H_2^2}{r_2^5} \right).
\end{align}
Applying the vector fields $X$ and $Y$ again to these expressions yields the second-order derivatives. Let $X_1=x+1$ and $X_2=x-1$. By explicit calculation, we obtain the non-mixed derivatives:
\begin{equation}
XXF_{\text{Taylor}} = XX(f_1) - XX(f_2), \quad YYF_{\text{Taylor}} = YY(f_1) - YY(f_2),
\end{equation}
where for $i=1, 2$,
\begin{align}
XX(f_i) &= \frac{13y^2}{r_i^3} - \frac{36y H_i X_i}{r_i^5} - \frac{9 H_i^2}{2 r_i^5} + \frac{45 H_i^2 X_i^2}{2 r_i^7}, \\
YY(f_i) &= \frac{13X_i^2}{r_i^3} + \frac{36 H_i y X_i}{r_i^5} - \frac{9 H_i^2}{2 r_i^5} + \frac{45 H_i^2 y^2}{2 r_i^7}.
\end{align}
The mixed derivatives are similarly derived as $XYF_{\text{Taylor}} = XY(f_1) - XY(f_2)$ and $YXF_{\text{Taylor}} = YX(f_1) - YX(f_2)$, with:
\begin{align}
XY(f_i) &= -\frac{13y X_i}{r_i^3} - \frac{6H_i}{r_i^3} + \frac{18H_i X_i^2 - 18y^2 H_i}{r_i^5} + \frac{45y X_i H_i^2}{2r_i^7}, \\
YX(f_i) &= -\frac{13y X_i}{r_i^3} + \frac{6H_i}{r_i^3} + \frac{18H_i X_i^2 - 18y^2 H_i}{r_i^5} + \frac{45y X_i H_i^2}{2r_i^7}.
\end{align}
Evaluating these exact components within \eqref{eq:meancurv} provides a completely closed-form rational approximation of $H^h$ applicable away from the singular locus.

\subsubsection{Numerical Approximation via Finite Differences}
To compute $H^h$ globally without relying on asymptotic limits, we employ numerical approximations of the horizontal gradient and Hessian. Standard finite difference schemes in Euclidean space must be adapted to traverse the integral curves of the non-holonomic left-invariant vector fields.

For a chosen grid spacing $\Delta > 0$, the first-order horizontal directional derivatives of $F$ at a generic point $p=(x,y,t)$ are approximated symmetrically by:
\begin{align}
XF(p) &\approx \frac{F(x + \Delta, y, t + 2y\Delta) - F(x - \Delta, y, t - 2y\Delta)}{2\Delta}, \\
YF(p) &\approx \frac{F(x, y + \Delta, t - 2x\Delta) - F(x, y - \Delta, t + 2x\Delta)}{2\Delta}.
\end{align}
Applying these modified difference operators recursively yields the discrete explicit approximations for the second-order terms. For example, the non-mixed second derivatives expand to:
\begin{align}
XXF(p) &\approx \frac{F(x + \Delta, y, t + 2y\Delta) - 2F(x,y,t) + F(x - \Delta, y, t - 2y\Delta)}{\Delta^2}, \\
YYF(p) &\approx \frac{F(x, y + \Delta, t - 2x\Delta) - 2F(x,y,t) + F(x, y - \Delta, t + 2x\Delta)}{\Delta^2}.
\end{align}
The mixed derivatives $XYF$ and $YXF$ follow by compounding the first-order operators. For example, $XYF(p) \approx \frac{YF(x + \Delta, y, t + 2y\Delta) - YF(x - \Delta, y, t - 2y\Delta)}{2\Delta}$ explicitly expands to the 4-point stencil:
\begin{align}
XYF(p) &\approx \frac{1}{4\Delta^2} \Big[ F(x+\Delta, y+\Delta, t + 2(y-x-\Delta)\Delta) - F(x+\Delta, y-\Delta, t + 2(y+x+\Delta)\Delta) \notag \\
&\quad - F(x-\Delta, y+\Delta, t - 2(y+x-\Delta)\Delta) + F(x-\Delta, y-\Delta, t - 2(y-x-\Delta)\Delta) \Big].
\end{align}
Similarly, we obtain the remaining cross derivative via $YXF(p) \approx \frac{XF(x, y+\Delta, t-2x\Delta) - XF(x, y-\Delta, t+2x\Delta)}{2\Delta}$:
\begin{align}
YXF(p) &\approx \frac{1}{4\Delta^2} \Big[ F(x+\Delta, y+\Delta, t - 2(x-y-\Delta)\Delta) - F(x-\Delta, y+\Delta, t - 2(x+y+\Delta)\Delta) \notag \\
&\quad - F(x+\Delta, y-\Delta, t + 2(x+y-\Delta)\Delta) + F(x-\Delta, y-\Delta, t + 2(x-y+\Delta)\Delta) \Big].
\end{align}
Substituting these into the sub-Riemannian divergence formula \eqref{eq:meancurv} allows for the global evaluation of $H^h$ across the bisector surface. This numerical scheme confirms the analytical expectation: the horizontal mean curvature steadily grows and strictly diverges to infinity as one approaches the characteristic locus where the horizontal gradient vanishes.

\section{Relationship with Complex Hyperbolic Bisectors}
In complex hyperbolic geometry, the boundary extension of a complex hyperbolic bisector (defined via the Bergman metric) naturally induces the Kor\'anyi gauge via Busemann functions \cite{Gol, PP06}. Thus, complex hyperbolic bisectors bound Kor\'anyi bisectors (spinal spheres) at infinity.

Carnot--Carath\'eodory bisectors generally do not share this relationship. Unlike the Kor\'anyi metric, $d_{cc}$ is rigidly governed by the non-holonomic constraints of the Heisenberg horizontal distribution, rather than arising as a direct geometric limit of the projective algebraic structure of ${\rm SU}(2,1)$ \cite{M02}. Because a purely transcendental surface cannot coincide with an algebraic spinal sphere, the surface $\mathcal{B}_{cc}^1$ cannot act as the boundary of a standard complex hyperbolic bisector. Therefore, Carnot--Carath\'eodory bisectors remain phenomena rigidly isolated to the intrinsic sub-Riemannian mechanics of the boundary itself.

\section*{Acknowledgement}
The second author was supported by National Research Foundation of Korea(NRF) granted by the Korea government(RS-2024-00345250)
The third author was funded by the Medicus programme, no. 83765.

\end{document}